\documentclass[12 pt]{amsart}
\usepackage{fullpage}
\usepackage{amsthm,amssymb,amsfonts,hyperref, mathrsfs}
\usepackage{url}
\hypersetup{
    colorlinks=true,       
    linkcolor=blue,          
    citecolor=cyan,        
}

\newtheorem{theorem}{Theorem}[section]
\newtheorem{lemma}[theorem]{Lemma}
\newtheorem{corollary}[theorem]{Corollary}
\newtheorem{proposition}[theorem]{Proposition}

\theoremstyle{remark}
\newtheorem{remark}{Remark}[section]

\makeatletter

\newcommand{\Rmnum}[1]{\expandafter\@slowromancap\romannumeral #1@}
\makeatother

\def\ri{\mathrm i}

\def\rb{\mathbb R}

\def\zb{\mathbb Z}
\def\cb{{\mathbb C}}

\def\rrw{\rightarrow}
\numberwithin{equation}{section}

\begin{document}
\title{On the distribution of rank statistic for strongly concave compositions}
\author{Nian Hong Zhou}
\address{School of Mathematical Sciences,  East China Normal University,
500 Dongchuan Road, Shanghai 200241, PR China}
\email{nianhongzhou@outlook.com}

\thanks{This research was supported by the National Science Foundation of China (Grant No. 11571114).}

\begin{abstract}
A strongly concave composition of $n$ is an integer partition with strictly
decreasing and increasing parts. In this paper we give a uniform asymptotic formula for the rank statistic of a strongly concave composition introduced by
Andrews, Rhoades and Zwegers ['Modularity of the concave composition generating function', Algebra \& Number Theory 7 (2013), no. 9, 2103--2139].
\end{abstract}

\maketitle

\section{Introduction}
A partition of a positive integer $n$ is a sequence of non-increasing positive integers whose sum equals $n$.
Let $p(n)$ be the number of integer partitions of $n$. To explain Ramanujan's famous partition
congruences with modulus $5$, $7$ and $11$, the rank and crank statistic for integer partitions was introduced by Dyson \cite{MR3077150}, Andrews  and Garvan \cite{MR929094, MR920146}.
Let $N(m, n)$ and $M(m,n)$ be the number of partitions of $n$ with rank $m$ and crank $m$, respectively. It is well known that
$$
\sum_{n\ge 0}N(m,n)q^n=\frac{1}{(q;q)_{\infty}}\sum_{n\ge 1}(-1)^{n-1}q^{n(3n-1)/2+|m|n}(1-q^n),
$$
and
$$
\sum_{n\ge 0}M(m,n)q^n=\frac{1}{(q;q)_{\infty}}\sum_{n\ge 1}(-1)^{n-1}q^{n(n-1)/2+|m|n}(1-q^n),
$$
where $(a;q)_{\infty}=\prod_{j\ge 0}(1-aq^j)$ for any $a\in\cb$ and $|q|<1$.

In \cite{MR1001259}, Dyson gave the following asymptotic formulae conjecture for the crank statistic for integer partitions:
\begin{equation}\label{dysc}
M\left( m,n \right)\sim \frac{\pi}{4\sqrt{6n}} {\rm sech}^2 \left( \frac{\pi m}{2\sqrt{6n}}   \right) p(n), n\rrw +\infty.
\end{equation}
Bringmann and Dousse \cite{MR3451872} proved \eqref{dysc} holds for all $|m|\le (\sqrt{n}\log n)/(\pi\sqrt{6})$.
Interestingly, Dousse and Mertens \cite{MR3337213} proved \eqref{dysc} also holds for $N(m,n)$. For more results on asymptotics
for rank and crank statistics for integer partitions, see \cite{MR3210725, MR3279269, MR3103192, MR3565363}.

A concave composition $\lambda$ is a nonnegative integer sequence $\{a_r\}_{r=1}^s$ of the form
$$a_1\ge a_2\ge \dots\ge a_{k-1}>a_k<a_{k+1}\le \dots\le a_{s-1}\le a_s,$$
for some $s\in\zb_+$, where $a_k$ is called the central part of
$\lambda$. If all the $"\ge"$ and $"\le"$ are replaced by $">"$ and $"<"$, respectively, we refer to a strongly
concave composition. The rank of $\lambda$ above is defined as ${\rm rk}(\lambda):=s-2k+1$, which analogs the rank statistic for integer partitions and measures the position
of the central part.

Let $\mathcal{V}(n)$ and $\mathcal{V}_d(n)$ be the set of all concave composition and all strongly concave composition, respectively, of nonnegative integer $n$. Also, let
$V(n)=\#\mathcal{V}(n)$ and $V_d(n)=\#\mathcal{V}_d(n)$ be the number of concave composition and strongly concave composition, respectively, of nonnegative integer $n$.  Andrews \cite{MR3048655} proved that
$$
v(q):=\sum_{n\ge 0}V(n)q^n=\sum_{n\ge 0}\frac{q^n}{(q^{n+1};q)_{\infty}^2}
$$
and
$$
v_d(q):=\sum_{n\ge 0}V_d(n)q^n=\sum_{n\ge 0}(-q^{n+1};q)_{\infty}^2q^n.
$$

Andrews, Rhoades and Zwegers \cite{MR3152010} proved that $v(q)$ is mixed modular form. More precise, they established the following modularity properties.
\begin{theorem}
Let $q=e^{2\pi\ri \tau}$ with $\tau\in\cb$ and $\Im(\tau)>0$. Define $f(\tau)=q(q;q)_{\infty}^3v(q)$ and
$$\hat{f}(\tau)=f(\tau)-\frac{\ri}{2}\eta(\tau)^3\int_{-\bar{\tau}}^{\ri\infty}\frac{\eta(z)^3}{(-\ri(z+\tau))^{1/2}}\,dz+\frac{\sqrt{3}}{2\pi\ri}\eta(\tau)\int_{-\bar{\tau}}^{\ri\infty}\frac{\eta(z)}{(-\ri(z+\tau))^{3/2}}\,dz,$$
where the Dedekind $\eta$-function is given by $\eta(\tau)=q^{1/24}(q;q)_{\infty}$. The function $\hat{f}$ transforms as a modular form of weight $2$ for ${\rm SL}_2(\zb)$.
\end{theorem}
For $v_d(q)$, Andrews \cite{MR3048655} proved that
$$
v_d(q)=2(-q;q)_{\infty}^2\sum_{n\ge 0}\left(\frac{-12}{n}\right)q^{\frac{n^2-1}{24}}-\sum_{n\ge 0}(-1)^nq^{\frac{n(n+1)}{2}},
$$
where $(\frac{\cdot}{\cdot})$ is the Kronecker symbol, that is, $v_d(q)+\sum_{n\ge 0}(-1)^nq^{\frac{n(n+1)}{2}}$ is essentially a modular function times a false theta function.
And so we are not expect that the properties of $V_d(n)$ as good as $V(n)$.
For example, \cite{MR3152010} point out that it is possible to obtain an asymptotic with a polynomial error for $V(n)$
by use a circle method of Bringmann and Mahlburg \cite{MR2823873}, but for $V_d(n)$, we can't seem to establish such asymptotic by existing approach.
Nevertheless, \cite[Theorem 1.5]{MR3152010} gave the following asymptotic expansion\footnote{It is need to note that the leading coefficient of the asymptotic expansion \eqref{asy1} is $2^{-1/4}3^{-5/4}$ rather that $2\cdot 2^{-1/4}3^{-5/4}$  in \cite[Theorem 1.5]{MR3152010}.}
\begin{equation}\label{asy1}
V_d(N)\sim 2^{-1/4}3^{-5/4}N^{-3/4}e^{2\pi\sqrt{\frac{N}{6}}}\left(1+\sum_{n\ge 1}c_nN^{-n/2}\right),
\end{equation}
for $N\rrw +\infty$, where $c_n\in\rb, n\in\zb_+$ is some computable constant.

Let $V_d(m,n)$ be the numbers of strongly concave compositions of $n$ with rank equal to $m$.
Andrews, Rhoades and Zwegers \cite{MR3152010} proved that
\footnote{We correct some sign error in \cite{MR3152010}, where in \eqref{eq0}, it is
$(-x;q)_{\infty}(-x^{-1}q;q)_{\infty}\sum_{n\ge 0}\left(\frac{-12}{n}\right)\dots$  rather than $ (x;q)_{\infty}(x^{-1}q;q)_{\infty}\sum_{n\ge 0}\left(\frac{12}{n}\right)\dots$.}
\begin{align}\label{eq0}
\sum_{n\ge 0}\sum_{m\in\zb}V_d(m,n)x^mq^n=&-\sum_{n\ge 0}(-1)^nq^{\frac{n(n+1)}{2}}x^{2n+1}\nonumber\\
&+(-x;q)_{\infty}(-x^{-1}q;q)_{\infty}\sum_{n\ge 0}\left(\frac{-12}{n}\right)x^{\frac{n-1}{2}}q^{\frac{n^2-1}{24}}.
\end{align}
In this paper we investigate the asymptotics for $V_d(m,n)$ as $n$ tends to infinity with arbitrary $m$, which is motivated by the questions in \cite[pp. 2108--2109]{MR3152010}
for the distribution of concave composition.\newline

The first result of this paper as follows.
\begin{proposition}\label{pr1}Let $p(n)$ be the number of integer partitions of nonnative integer $n$ and let $p(-\ell)=0$ for $\ell\in\zb_+$. Then for $N,\ell\in\zb$ we have
\begin{equation}\label{eq1}
V_d\left(\ell,N+\frac{|\ell|(|\ell|+1)}{2}\right)=\sum_{n\ge 0}\left(\frac{-3}{2n+1}\right)p\left(N-\frac{2n(n+1)}{3}-n|\ell|\right).
\end{equation}
In particular,  for $m,n\in\zb$ with $0\le n<\frac{|m|(|m|+5)}{2}+4$,
\begin{equation}\label{eq11}
V_d(m,n)=p\left(n-\frac{|m|(|m|+1)}{2}\right).
\end{equation}
\end{proposition}
From the above Proposition \ref{pr1} we prove the following uniform asymptotics for $V_d(m,n)$ as $n\rrw +\infty$.
\begin{theorem}\label{th1}We have uniformly for all $\ell\in\zb$ and $N\rrw +\infty$,
\begin{equation}\label{eqm0}
V_d\left(\ell,N+\frac{|\ell|(|\ell|+1)}{2}\right)= p(N)F\left(\frac{\pi|\ell|}{\sqrt{6N}}\right)\left(1+O\left(N^{-1/10}\right)\right),
\end{equation}
where the implied constant is absolute and
$$F(\alpha)=\frac{1+e^{-\alpha}}{1+e^{-\alpha}+e^{-2\alpha}}.$$
In particular, if integer $m=o\left(N^{3/8}\right)$ then
\begin{equation}\label{nnn}
\frac{V_d\left(m,N\right)}{V_d(N)}\sim \frac{1}{(24N)^{1/4}}\exp\left(-\frac{\pi m^2}{\sqrt{24N}}\right).
\end{equation}
\end{theorem}

We are going to statement the last one of results of this paper. Let the real function $\Psi_d(x)$ be defined as:
\begin{equation}
\Psi_d(x)=\lim_{N\rrw +\infty}\frac{1}{V_d(N)}\#\left\{\lambda \in \mathcal{V}_d(N): \frac{{\rm rk}(\lambda)}{(6N/\pi^2)^{1/4}}\le x\right\},
\end{equation}
for $x\in\rb$, which gives a limiting distribution of the rank statistic for
strongly concave composition. It is clear that
\begin{align*}
\Psi_d(x)=\lim_{N\rrw +\infty}\frac{1}{V_d(N)}\sum_{\substack{m\in\zb\\ m\le ({6N}/{\pi^2})^{1/4}x}}\sum_{\substack{\lambda \in \mathcal{V}_d(N)\\ {\rm rk}(\lambda)=m}}1
=\lim_{N\rrw +\infty}\sum_{\substack{m\in\zb\\ m\le ({6N}/{\pi^2})^{1/4}x}}\frac{V_d(m,N)}{V_d(N)},
\end{align*}
$$\Psi_d(-\infty)=0 \;\mbox{and}\; \Psi_d(+\infty)=1.$$
Hence by using \eqref{nnn} and the fact that $V_d(m,N)=V_d(|m|,N)$, it is easy to deduce the following corollary by Abel's summation formula.
\begin{corollary}\label{cor1} We have for any fixed $x\in\rb$, $\Psi_d(x)$ is the distribution function of the standard normal distribution, that is,
$$\Psi_d(x)=\frac{1}{\sqrt{2\pi}}\int_{-\infty}^{x}e^{-\frac{x^2}{2}}\,dx.$$
\end{corollary}
\begin{remark}
The Corollary \ref{cor1} is suggested by the referee.
\end{remark}
\section{Proofs of results}
\subsection{The proof of Proposition \ref{pr1}}~

By the Jacobi triple product formula
$$(q;q)_{\infty}(-xq;q)_{\infty}(-x^{-1};q)_{\infty}=\sum_{n\in\zb}q^{\frac{n(n+1)}{2}}x^n$$
and the basic properties of Kronecker symbol we obtain that
\begin{align*}
\sum_{\substack{n\ge 0\\ m\in\zb}}V_d(m,n)x^mq^n
=&-\sum_{n\ge 0}(-1)^nq^{\frac{n(n+1)}{2}}x^{2n+1}\\
&+\frac{1}{(q;q)_{\infty}}\sum_{\ell\in\zb}q^{\frac{\ell(\ell+1)}{2}}x^{-\ell}\sum_{n\ge 0}\left(\frac{-
12}{2n+1}\right)x^{n}q^{\frac{n(n+1)}{6}}.
\end{align*}
by using \eqref{eq0}. This yields for integer $r\ge 0$,
\begin{align*}
\sum_{n\ge 0}V_d(-r,n)q^n&=\frac{1}{(q;q)_{\infty}}\sum_{\substack{\ell-n=r\\ \ell\in\zb, n\ge 0}}\left(\frac{-3}{2n+1}\right)\left(\frac{2}{2n+1}\right)^2q^{\frac{n(n+1)}{6}+\frac{\ell(\ell+1)}{2}}\\
&=\sum_{k\ge 0}p(k)q^{k}\sum_{ n\ge 0}\left(\frac{-3}{2n+1}\right)q^{\frac{n(n+1)}{6}+\frac{(n+r)(n+r+1)}{2}}\\
&=\sum_{N\ge 0}q^N\sum_{n\ge 0}\left(\frac{-3}{2n+1}\right)p\left(N-\frac{2n(n+1)}{3}-rn-\frac{r(r+1)}{2}\right).
\end{align*}
Which means that
$$V_d\left(-\ell,N+\frac{\ell(\ell+1)}{2}\right)=\sum_{n\ge 0}\left(\frac{-3}{2n+1}\right)p\left(N-\frac{2n(n+1)}{3}-n\ell\right)$$
for all integer $\ell\ge 0$. Thus by the fact that $V_d(-m,n)=V_d(m,n)$ we get the proof of \eqref{eq1} in Proposition \ref{pr1}. Further, if $2\ell+4>N$ then
$$V_d\left(\ell,N+\frac{|\ell|(|\ell|+1)}{2}\right)=p(N).$$
Namely we get the proof of \eqref{eq11} in Proposition \ref{pr1}.

\subsection{Hardy--Ramanujan for $p(n)$ and its applications}~

We need the following Hardy--Ramanujan asymptotic result for $p(n)$, which can be found in \cite{MR1575586}.
\begin{lemma}\label{lem1}We have for $n\in\zb_+$,
\begin{equation*}
p(n)-\hat{p}(n-1/24)= O\left(n^{-1}e^{B\sqrt{n}/2}\right),
\end{equation*}
where $B=2\pi/\sqrt{6}$ and
\[\hat{p}(x)=\frac{e^{B\sqrt{x}}}{4\sqrt{3}x}\left(1-\frac{1}{B\sqrt{x}}\right)\]
(throughout, $B$ and $\hat{p}(x)$ be defined as above).
\end{lemma}

We also need the following approximation for $p(X+r)$ with $r=o(X^{3/4})$.
\begin{lemma}\label{lem2}
Let $r=o(X^{3/4})$ and $X$ be sufficiently large, we have
$$\frac{p(X+r)}{p(X)}=e^{\frac{Br}{2\sqrt{X}}}\left(1+O\left(\frac{1}{X}+\frac{|r|}{X}+\frac{|r|^2}{X^{3/2}}\right)\right).$$
\end{lemma}
\begin{proof}
From Lemma \ref{lem1}, it is clear that
\begin{align*}
\frac{\hat{p}(X+r)}{\hat{p}(X)}&=e^{B\left(\sqrt{X+r}-\sqrt{X}\right)}\left(1+O\left(\frac{|r|}{X}\right)\right)\\
&=e^{\frac{Br}{2\sqrt{X}}+O(r^2/X^{3/2})}\left(1+O\left(\frac{|r|}{X}\right)\right)\\
&=e^{\frac{Br}{2\sqrt{X}}}\left(1+O\left(\frac{|r|}{X}+\frac{|r|^2}{X^{3/2}}\right)\right)
\end{align*}
by the generalized binomial theorem.
Note that
$$
\frac{p(N)}{\hat{p}(N)}=1+O\left(\frac{1}{N}\right)
$$
holds for all $N\ge 1$  we obtain that
$$\frac{p(X+r)}{p(X)}=e^{\frac{Br}{2\sqrt{X}}}\left(1+O\left(\frac{1}{X}+\frac{|r|}{X}+\frac{|r|^2}{X^{3/2}}\right)\right).$$
Which completes the proof of the lemma.
\end{proof}
\subsection{The proof of Theorem \ref{th1}}
\subsubsection{Case  $|\ell| >\sqrt{N}(\log N)^2$.}
We first denote
\begin{equation}\label{eq20}
F(\ell,N):=V_d\left(\ell,N+\frac{|\ell|(|\ell|+1)}{2}\right).
\end{equation}
For $N/2 \ge |\ell| >\sqrt{N}(\log N)^2 $, from Proposition \ref{pr1} and Lemma \ref{lem2} we obtain that
\begin{align*}
F(\ell,N)&=\sum_{\substack{n\ge 0\\ 2n(n+1)/{3}+n|\ell|\le N}}\left(\frac{-3}{2n+1}\right)p\left(N-\frac{2n(n+1)}{3}-n|\ell|\right)\\
&=p(N)+O\left(\sum_{\substack{n\ge 2\\ 2n(n+1)/{3}+n\ell\le N}}p\left(N-n|\ell|\right)\right)\\
&=p(N)+O\left(\sqrt{N}p(N-2|\ell| )\right)=p(N)+O\left(\sqrt{N}p(N- \lfloor \sqrt{N}(\log N)^2\rfloor)\right)\\
&=p(N)\left(1+O\left(\sqrt{N}\exp\left(-\frac{B\lfloor \sqrt{N}(\log N)^2\rfloor}{2\sqrt{N}}\right)\right)\right),
\end{align*}
where $\lfloor\cdot\rfloor$ is the greatest integer function. Hence, we have if $N/2 \ge |\ell| >\sqrt{N}(\log N)^2$ then
\begin{equation}\label{eq21}
F(\ell, N)=p(N)\left(1+O\left(N^{-\sqrt{\log N}}\right)\right).
\end{equation}

\subsubsection{Case  $|\ell| \le\sqrt{N}(\log N)^2$.}
If $0\le \ell\le \sqrt{N}(\log N)^2$ since
$$\left(\frac{-3}{2n+1}\right)=\begin{cases} ~~1 \quad &if~ n\equiv 0\bmod 3,\\
~~0 &if~ n\equiv 1\bmod 3,\\
-1 &if~ n\equiv 2\bmod 3,
\end{cases}$$
we have
\begin{align*}
F(\ell,N)=&\sum_{n\ge 0}\left(\frac{-3}{2n+1}\right)p\left(N-\frac{2n(n+1)}{3}-n\ell\right)\\
=&\sum_{n\ge 0}\left[p\left(N-Q_1(n,\ell)\right)-p\left(N-Q_2(n,\ell)\right)\right],
\end{align*}
where
$$Q_1(n,\ell)=2n(3n+1)+3n\ell\;\mbox{and}\; Q_2(n,\ell)=Q_1(n,\ell)+(8n+4+2\ell).$$
We split that
\begin{align*}
\frac{F(\ell,N)}{p(N)}=
&\frac{1}{p(N)}\sum_{\substack{n\ge 0\\n^2+n\ell > \sqrt{N}(\log N)^2}}\left[p\left(N-Q_1(n,\ell)\right)-p\left(N-Q_2(n,\ell)\right)\right]\\
&+\frac{1}{p(N)}\sum_{\substack{n\ge 0\\ n^2+n\ell \le \sqrt{N}(\log N)^2}}\left[p\left(N-Q_1(n,\ell)\right)-p\left(N-Q_2(n,\ell)\right)\right]
=:R+I.
\end{align*}
Noting that $Q_2(n,\ell)\ge Q_1(n,\ell)\ge n^2+n\ell$ for all $n\ge 0$ we estimate that
\begin{align*}
|R|&\le  \frac{2}{p(N)}\sum_{\substack{n\ge 0\\n^2+n\ell > \sqrt{N}(\log N)^2}}p\left(N-Q_1(n,\ell)\right)\\
&\le \frac{2}{p(N)}\sum_{\substack{n\ge 0\\ n^2+n\ell > \sqrt{N}(\log N)^2}}p\left(N-(n^2+n\ell)\right)\le 2\sqrt{N}\frac{p(N-\lfloor \sqrt{N}(\log N)^2\rfloor)}{p(N)}.
\end{align*}
Thus Lemma \ref{lem2} implies the following estimate for $R$,
\begin{equation*}
R\ll \sqrt{N}e^{-\frac{B\lfloor \sqrt{N}(\log N)^2\rfloor}{2\sqrt{N}}}\ll N^{-\sqrt{\log N}}.
\end{equation*}

To estimate $I$, we note that
$$0\le Q_1(n,\ell)\le Q_2(n,\ell)\le 16(n^2+n\ell)+2\ell+4=O\left(\sqrt{N}(\log N)^2\right),$$
if $n\ge 0$ and $n^2+n\ell\le \sqrt{N}(\log N)^2$. Then by Lemma \ref{lem2} we have
\begin{align*}
I=&\sum_{n\ge 0}\left(e^{-\frac{BQ_1(n,\ell)}{2\sqrt{N}}}-e^{-\frac{BQ_2(n,\ell)}{2\sqrt{N}}}\right)-\sum_{\substack{n\ge 0\\ n^2+n\ell > \sqrt{N}(\log N)^2}}
\left(e^{-\frac{BQ_1(n,\ell)}{2\sqrt{N}}}-e^{-\frac{BQ_2(n,\ell)}{2\sqrt{N}}}\right)\\
&+O\left(\sum_{i=1}^2\sum_{\substack{n\ge 0\\ n^2+n\ell \le \sqrt{N}(\log N)^2}}e^{-\frac{BQ_i(n,\ell)}{2\sqrt{N}}}\left(\frac{1}{N}+\frac{Q_i(n,\ell)}{N}+
\frac{Q_i(n,\ell)^2}{N^{3/2}}\right)\right)=I_M+I_R
\end{align*}
with
\begin{equation*}
I_M=\sum_{n\ge 0}\left(e^{-\frac{BQ_1(n,\ell)}{2\sqrt{N}}}-e^{-\frac{BQ_2(n,\ell)}{2\sqrt{N}}}\right)
\end{equation*}
and
\begin{align*}
I_R\ll &\sum_{\substack{n\ge 0\\ n^2+n\ell > \sqrt{N}(\log N)^2}}e^{-\frac{B(n^2+\ell n)}{\sqrt{N}}}
+\sum_{\substack{n\ge 0\\ n^2+n\ell\le \sqrt{N}(\log N)^2}}\frac{(\log N)^{4}}{N^{1/2}}e^{-\frac{B(n^2+n\ell)}{2\sqrt{N}}}\\
\ll &N^{-\sqrt{\log N}}+N^{-1/2}(\log N)^4\sum_{\substack{n\ge 0\\ n^2+n\ell\le \sqrt{N}(\log N)^2}}1\ll N^{-1/4}(\log N)^5.
\end{align*}
From the above, we conclude above that
\begin{equation}\label{eq24}
{F(\ell,N)}/{p(N)}=I_M+O\left(N^{-1/5}\right)
\end{equation}
holds for $0\le \ell\le \sqrt{N}(\log N)^2$.

For the estimate of $I_M$ we need the following lemma.
\begin{lemma}\label{lem4}Let $0\le \ell= o(\alpha^{-1})$, then as $\alpha\rrw 0^+$,
$$f(\alpha):=\alpha\sum_{n\ge 0}(4n+\ell)e^{-2\alpha n^2-\alpha n\ell}=1+O\left(\sqrt{\alpha}+|\alpha\ell|\right).$$
\end{lemma}
\begin{proof}By Abel's summation formula or integration by parts for a Riemann--Stieltjes integral we obtain that
\begin{align*}
f(\alpha)&=4\alpha\sum_{n\ge 0}(n+\ell/4)e^{-2\alpha(n+\ell/4)^2+\frac{\alpha\ell^2}{8}}\\
&=4\alpha e^{\frac{\alpha\ell^2}{8}}\int_{0-}^{\infty}e^{-2\alpha(x+\ell/4)^2}\,d\left(\sum_{0\le n\le x}(n+\ell/4)\right)\\
&=4\alpha e^{\frac{\alpha\ell^2}{8}}\left(\int_{0}^{\infty}e^{-2\alpha(x+\ell/4)^2}\,d \left(\frac{x^2}{2}+\frac{x\ell}{4}\right)
+O\left(\alpha\int_{0}^{\infty}(x+\ell/4)^2e^{-2\alpha(x+\ell/4)^2}\,d x\right)\right)\\
&=4\alpha e^{\frac{\alpha\ell^2}{8}}\left(\int_{\ell/4}^{\infty}xe^{-2\alpha x^2}\,dx+O\left(\alpha\int_{\ell/4}^{\infty}x^2e^{-2\alpha x^2}\,dx\right)\right)\\
&=e^{\frac{\alpha\ell^2}{8}}\int_{{\alpha\ell^2}/{8}}^{\infty}e^{-x}\,dx
+O\left(\sqrt{\alpha}e^{\frac{\alpha\ell^2}{8}}\int_{{\alpha\ell^2}/{8}}^{\infty}x^{1/2}e^{-x}\,dx\right)=1+O\left(\sqrt{\alpha}+|\alpha\ell|\right),
\end{align*}
which completes the proof of the lemma.
\end{proof}
We now evaluate $I_M$. By the definition of $F(\alpha)$ and $I_M$, it is clear that if $\ell\ge N^{3/8}$ then
\begin{align*}
I_M&=\sum_{0\le n\le N^{1/5}}\left(e^{-\frac{BQ_1(n,\ell)}{2\sqrt{N}}}-e^{-\frac{BQ_2(n,\ell)}{2\sqrt{N}}}\right)+O(N^{-\sqrt{\log N}})\\
&=\sum_{0\le n\le N^{1/5}}e^{-\frac{B(3n+1)n}{\sqrt{N}}}\left(1-e^{-\frac{B(\ell+4n)}{\sqrt{N}}}\right)e^{-\frac{3Bn\ell}{2\sqrt{N}}}+O(N^{-\sqrt{\log N}})\\
 &=\left(1+O(N^{-1/10})\right)\sum_{0\le n\le N^{1/5}}\left(1-e^{-\frac{B(\ell+4n)}{\sqrt{N}}}\right)e^{-\frac{3Bn\ell}{2\sqrt{N}}}+O(N^{-\sqrt{\log N}})\\
&= \left(1+O(N^{-1/10})\right)\frac{1-e^{-\frac{B\ell}{\sqrt{N}}}}{1-e^{-\frac{3B\ell}{2\sqrt{N}}}}=\left(1+O(N^{-1/10})\right)F\left(\frac{B\ell}{2\sqrt{N}}\right)
\end{align*}
and if $0\le \ell\le N^{3/8}$ then
\begin{align*}
I_M=&\sum_{ 0\le n\le N^{2/5}}\left(e^{-\frac{Bn}{\sqrt{N}}}-e^{-\frac{B(5n+\ell)}{\sqrt{N}}}\right)e^{-\frac{B(6n^2+3n\ell)}{2\sqrt{N}}}+O(N^{-\sqrt{\log N}})\\
=& \left(1+O(N^{-1/10})\right)\sum_{ 0\le n\le N^{2/5}}\frac{B(4n+\ell)}{\sqrt{N}}e^{-\frac{B(6n^2+3n\ell)}{2\sqrt{N}}}+O(N^{-\sqrt{\log N}})\\
=&\left(1+O(N^{-1/10})\right)\frac{B}{\sqrt{N}}\sum_{n\ge 0}(4n+\ell)e^{-\frac{B(6n^2+3n\ell)}{2\sqrt{N}}}+O(N^{-\sqrt{\log N}})\\
=&\frac{2}{3}\left(1+O(N^{-1/10})\right)\left(1+O(N^{-1/4}+\ell N^{-1/2})\right)= \left(1+O(N^{-1/10})\right)F\left(\frac{B\ell}{2\sqrt{N}}\right)
\end{align*}
by the use of Lemma \ref{lem4}. Thus it is clear that for $0\le \ell\le \sqrt{N}(\log N)^2$,
\begin{equation}\label{eq25}
F(\ell, N)=p(N)F\left(\frac{\pi\ell}{\sqrt{6N}}\right)\left(1+O(N^{-1/10})\right)
\end{equation}
by use \eqref{eq24} and $B=2\pi/\sqrt{6}$.

Finally,  by using \eqref{eq20}, \eqref{eq21}, \eqref{eq25} and the fact that $V_d(m,n)=V_d(|m|,n)$ we finish the proof of \eqref{eqm0}. By using \eqref{asy1}, \eqref{eqm0} and Lemma \ref{lem1} we obtain the proof of \eqref{nnn}, which completes the proof of Theorem \ref{th1}.
\section*{Acknowledgment}

The author would like to thank the referee for very helpful and detailed comments and suggestions.

\end{document}